\documentclass[12pt,a4paper,fleqn]{article}
\usepackage{amsmath}
\usepackage{amsthm}
\usepackage{amssymb}
\usepackage{amsfonts}
\usepackage{enumerate}
\usepackage{graphicx}
\usepackage{color}
\usepackage{rotating}
\usepackage{stmaryrd}
\newtheorem{theorem}{Theorem}[section]
\newtheorem*{theorem*}{Theorem}
\newtheorem{lemma}[theorem]{Lemma}
\newtheorem*{lemma*}{Lemma}
\newtheorem{corollary}[theorem]{Corollary}
\newtheorem*{corollary*}{Corollary}
\newtheorem{proposition}[theorem]{Proposition}
\newtheorem*{proposition*}{Proposition}

{\theoremstyle{definition}
\newtheorem{definition}[theorem]{Definition}
\newtheorem*{definition*}{Definition}

\newtheorem*{example*}{Example}
\newtheorem*{examples*}{Examples}
\newtheorem{remark}[theorem]{Remark}
\newtheorem*{remark*}{Remark}
\newtheorem{result}{Result}

\newtheorem*{question*}{Question}

\newtheorem*{problem*}{Problem}

\newtheorem*{notation*}{Notation}

\newtheorem*{note*}{Note}

\newtheorem*{algorithm*}{Algorithm}
}
\newcommand{\Colon}{\:\mathbin{:}\:}

\newcommand{\Bref}[1]{(\ref{#1})}
\newcommand{\shape}{\mathbin{\mbox{\rm sh}}\,}
\newcommand{\ic}{\mathbin{{\rm ic}}}
\newcommand{\oc}{\mathbin{{\rm oc}}}

\newcommand{\LEQ}{\leqslant}

\let\unlhd\trianglelefteqslant
\title{On $C$-bases, partition pairs and filtrations for\\ induced or restricted Specht modules} 
\author{
Christos A. Pallikaros%
\thanks{%
\textit{Department of Mathematics and Statistics,
University of Cyprus,
P.O.Box 20537,
1678 Nicosia,
Cyprus.}
{E-mail: pallikar@ucy.ac.cy}
}
}
\date{December 11, 2017}
\begin{document}
\maketitle
\begin{abstract}
We obtain alternative explicit Specht filtrations for the induced and the restricted Specht modules in the Hecke algebra of the symmetric group (defined over the ring $A=\mathbb Z[q^{1/2},q^{-1/2}]$ where $q$ is an indeterminate) using $C$-bases for these modules.
Moreover, we provide a link between a certain $C$-basis for the induced Specht module and the notion of pairs of partitions.
\\[1.5ex]
\end{abstract}
\section{Introduction}\label{sec:1c}

For any Coxeter system $(W,S)$ Kazhdan and Lusztig~\cite{KLu79} introduced three preorders $\LEQ_L$, $\LEQ_R$ and $\LEQ_{LR}$
with corresponding equivalence relations $\sim_L$, $\sim_R$ and $\sim_{LR}$,
the equivalence classes of which are called left cells, right cells and two-sided cells respectively.
From their construction, these cells give rise to corresponding cell modules for the Hecke algebra $\mathcal H$
(defined over the ring $A=\mathbb Z[q^{1/2},q^{-1/2}]$ where $q$ is an indeterminate)
associated with this Coxeter system.
In the special case $W$ is the symmetric group, the cell to which an element belongs can be determined by examining the tableaux resulting
from an application of the Robinson-Schensted process to that element.
In particular, elements belonging to the same right cell have the same recording tableau.

In~\cite{KLu79} Kazhdan and Lusztig introduced the $C$-basis for~$\mathcal H$ and this basis turns out to be extremely useful
when investigating the representation theory of~$\mathcal H$ (or $W$).
In particular, each cell of~$W$ provides an integral representation of~$W$.

In~\cite{BVo83}, Barbasch and Vogan considering the case $W$ is a Weyl group, showed that the $C$-basis is compatible with induction and restriction
of representations and established very useful rules for naturally decomposing such representations.
These results were generalized by Roichman~\cite{Roi98} and Geck~\cite{Gec03}.
In~\cite{MPa17} a description of how one can identify the right cells occurring in an induced or restricted Kazhdan-Lusztig cell via their recording tableau is given.
In view of this, also making use of the results in Geck~\cite{Gec03} where the approach is independent of the theory of primitive ideals,
we obtain elementary constructions of  Kazhdan-Lusztig cell module filtrations (and hence Specht filtrations) for the induced or restricted  cell modules of the Hecke algebra
of $W$ when $W$ is the symmetric group $S_m$.
In~\cite{MPa05} it was observed that a suitable choice of a subset of a $C$-basis of~$\mathcal H$ gives in a natural way
a corresponding $C$-basis for the Specht module.
In Section~\ref{Section:SpechtFiltrations} we obtain, by similar argument, a $C$-basis for the induced Specht module.
These observations enable us to obtain alternative explicit Specht filtrations for the induced and the restricted Specht module (see Theorems~\ref{thm:SpFilForS2H} and~\ref{thm:SpFilForSlambdaRestToH}).
These filtrations have the following property:
Denoting by~$\nu^{(i)}$ the partition associated to the $i$th factor of the filtration for the induced or restricted Specht module,
we get that $\nu^{(i)}\vartriangleleft\nu^{(i+1)}$  for all~$i$.

We note here that corresponding filtrations for the symmetric group case already appear in James~\cite{Jam78} and can be deduced from the results in~\cite{DJa86} for the Hecke algebra of the symmetric group case.
Also in Jost~\cite{Jost} branching theorems for Specht modules in the Hecke algebra are obtained.
Moreover, in~\cite{Mathas09} Mathas gives Specht filtrations for the induced Specht module working in the more general context of Ariki-Koike algebras (see also~\cite{Mathas16} for the restricted Specht module case).

In Section~\ref{sec:LinkSequenPairsOfPart} we provide a link between a $C$-basis of the induced Specht module described earlier on in the paper with the notion of pairs of partitions introduced by James in~\cite{James1977}.
We remark here that in~\cite{James1977,Jam78}, see also~\cite{DJa86} for the Hecke algebra case, the notion of pairs of partitions is used in order to construct Specht filtrations for certain modules which are generalizations of induced Specht modules associated to partitions.
In particular, in this final section, we associate to each pair of partitions $(\lambda,\mu)$ for $m$  (note that this terminology allows $\mu$ to be a composition whereas $\lambda$ is necessarily a partition) a subset $L(\lambda,\mu)$ of the set $L(\mu)=\{w\in W:\ w\le_L w_{J(\mu)}\}$ where $w_{J(\mu)}$ is the longest element in the standard parabolic subgroup $W_{J(\mu)}$ of $W=S_m$ corresponding to the composition $\mu$ of $m$.
(Recall that $L(\mu)=\mathfrak X^{-1}_{J(\mu)}w_{J(\mu)}$ where $\mathfrak X_{J(\mu)}$ is the set of distinguished right coset representatives of $W_{J(\mu)}$ in $W$.)
We show that $L(\lambda,\mu)$ is a union of left cells of $W$ (see Proposition~\ref{prop:LmulambdaLlambdamu}), and relate this set to a certain $C$-basis for the induced Specht module obtained in Section~\ref{Section:SpechtFiltrations}.
Moreover, in Theorem~\ref{prop:(1)} (see also Corollary~\ref{cor:4.9}) we show that the set $\{e\in \mathfrak X_{J(\mu)}:\ e^{-1}w_{J(\mu)}\in L(\lambda,\mu)\}$ satisfies the Schreier property.
A key observation in proving some of the above results is given in Proposition~\ref{prop:4.2}  which relates the cells occurring in the set $\{w\in W:\ w\le_Lw_{J(\mu)}\}$ to semistandard tableaux of type $\mu$.

\section{Generalities and preliminaries}\label{sec:PrelGen}

\subsection{Hecke algebras: background}\label{Subsection:HeckeABackgr}

In this subsection, we outline some terminology and classical results of
Hecke algebra theory, with appropriate notation.

For basic concepts relating to Coxeter groups and Hecke algebras,
see Geck and Pfeiffer~\cite{GPf00}, Kazhdan and Lusztig~\cite{KLu79}
and Humphreys~\cite{Hum90}.
In particular, for a Coxeter system $(W,S)$,
the \emph{length} $l(w)$ $(=l_S(w))$ of an element $w\in W$ with
respect to the generators $S$ is the length of the word in $S$ with
fewest generators which is equal to $w$,
$W_J=\langle J\rangle$ denotes the parabolic subgroup determined by a
subset $J$ of $S$, $w_J$ denotes the longest element of $W_J$
and $l_J$ is its length function
(in fact, $l_J(w)=l(w)$ for all $w\in W_J$),
$\mathfrak{X}_J$ denotes the set of minimum length elements in the
right cosets of $W_J$ in $W$
(the distinguished right coset representatives),
$\LEQ$ denotes the strong Bruhat order on $W$,
and $w<w'$ means $w\LEQ w'$ and $w\neq w'$ if $w,w'\in W$.
The pair $(W_J,J)$ is a Coxeter system whose length function is the same as the
restriction of the length function of $(W,S)$ to it;
consequently, $w_J$ is determined entirely by $J$.
\par
The Hecke algebra $\mathcal{H}$ corresponding to $(W,S)$ and
defined over the ring $A=\mathbf{Z}[q^{\frac{1}{2}},q^{-\frac{1}{2}}]$,
where $q$ is an indeterminate, has a free $A$-basis
$\{T_w\Colon w\in W\}$ and multiplication defined by the rules
\begin{equation}
\label{eqn:1a}
\begin{tabular}{rl}
(i) & $T_wT_{w'}=T_{ww'}$ if $l(ww')=l(w)+l(w')$ and \\[2pt]
(ii) & $(T_s+1)(T_s-q)=0$ if $s\in S$.
\end{tabular}
\end{equation}
The basis $\{T_w\Colon w\in W\}$ is called the \emph{$T$-basis} of
$\mathcal{H}$.
(See \cite{KLu79}).

\begin{result}
[{\cite[Theorem~1.1]{KLu79}}]
\label{res:1a}
$\mathcal{H}$ has a basis $\{C_w\Colon w\in W\}$, the \emph{$C$-basis},
whose terms have the form
$C_y =
\sum_{x\LEQ y}
(-1)^{l(y)-l(x)}q^{\frac{1}{2}l(y)-l(x)}P_{x,y}(q^{-1})T_x$,
where $P_{x,y}(q)$ is a polynomial in $q$ with integer coefficients
of degree $\le \frac{1}{2}\left(l(y)-l(x)-1\right)$ if $x<y$ and
$P_{y,y}=1$.
\end{result}
\par
There is a ring automorphism $j_{\mathcal{H}}$
of $\mathcal{H}$ defined by
$\left(\sum_{y\in W}a_yT_y\right)\!\!j_{\mathcal{H}}$
$=\sum_{y\in W}\overline{a}_y\left(-q^{-1}\right)^{l(y)}T_y$,
where $a\mapsto \overline{a}$ is the automorphism of $A$ defined by
$q^{\frac{1}{2}}\mapsto q^{-\frac{1}{2}}$
(this is the ring involution $\jmath$ in \cite[p.166]{KLu79}).
This automorphism is used to relate the $C$-basis of $\mathcal{H}$ to
another basis $\{C_w'\colon w\in W\}$ known as the $C'$-basis, which
may be defined by $C_w'=(-1)^{l(w)}C_wj_{\mathcal{H}}$.
\par
There is another $A$-basis of $\mathcal{H}$ which has also been used
in the literature instead of the $T$-basis and sometimes called the
$\tilde T$-basis. 
This is defined by Lusztig~\cite{Lus83} as $\{\tilde T_w: w\in W\}$
where $\tilde T_w=(q^{-1/2})^{l(w)}T_w$ ($w\in W$).
In~\cite{Gec03} Geck uses the $\tilde T$-basis but calls it the
$T$-basis.
\par
Moreover, with $q$ as in Result~\ref{res:1a}, and using the approach
of Lusztig~\cite{Lus83} with $q_s=q$ for all $s\in S$, let $\Gamma$ be
the (multiplicative) abelian group generated by $q^{1/2}$ and consider
the natural total order $\LEQ_\Gamma$ on $\Gamma$ given by
$\Gamma_+=\{(q^{1/2})^m: m>0\}$.
Then $q^{1/2}\in \Gamma_+$ as required in~\cite{Lus83,Gec03} and
$q^{i/2}\LEQ_\Gamma q^{j/2}$ if and only if $i\le j$---compare~\cite[pages
105--106]{Lus83}.
The $C'$-basis in~\cite{Lus83} derived using this total order is
precisely the $C'$-basis in~\cite{KLu79} and described in
Result~\ref{res:1a} (see~\cite[pages 101,105--106]{Lus83}).
So with the above order on $\Gamma$ the $C$-basis in~\cite{Gec03} is,
in fact, the $C'$-basis in~\cite{KLu79} (and in this paper).
\par
In what follows we would like to use some results from~\cite{Gec03}.
The preceding remarks allow us to translate directly from results
in~\cite{Gec03} to our notation.
For example, applying the ring involution $\jmath_{\mathcal{H}}$,
we see that~\cite[Lemma~2.2 and Corollary~3.4]{Gec03} still hold for
our $C$-basis and $T$-basis.
For convenience, we record the result corresponding to $C'$ obtained
by applying $\jmath_{\mathcal{H}}$ to Result~\ref{res:1a}.
\begin{result}
[{\cite{KLu79}}]
\label{res:2a}
$\mathcal{H}$ has a basis $\{C'_w\Colon w\in W\}$, the
\emph{$C'$-basis}, whose terms have the form
$C'_y = q^{-\frac{1}{2}l(w)}\sum_{x\LEQ y}P_{x,y}(q)T_x$,
where $P_{x,y}(q)$ is a polynomial in $q$ with integer coefficients
of degree $\le \frac{1}{2}\left(l(y)-l(x)-1\right)$ if $x<y$ and
$P_{y,y}=1$.
\end{result}
We have an immediate corollary:
\begin{corollary}
For any $y\in W$ and $h\in\mathcal{H}$,
\begin{equation}
\label{eqn:2a}
C_{y}'h=\sum_{x\in W}\alpha_{y,h,x}C_{x}'
\end{equation}
for suitable elements $\alpha_{y,h,x}\in A$.
\end{corollary}

We now introduce the notation $\to_s$ for a relation defined by
$y\to_s x$ if $x,y\in W$, $s\in S$ and $\alpha_{y,T_s,x}\neq0$.
This relation and others appearing in this paragraph have been
described in \cite{KLu79,Lus83}, though the notation may differ.
The preorder $\LEQ_{R}$ on $W$ is defined by $u\LEQ_{R} v$ if there
is a sequence $t_1,\ldots,t_m$ of elements of $S$, not necessarily
distinct, and a sequence $u=u_0,u_1,\ldots,u_m=v$ such that
$u_i\to_{t_i}u_{i-1}$ for $i=1,\ldots,m$.
The corresponding equivalence relation is $\sim_R$.
The preorder $\LEQ_{L}$ and equivalence relation $\sim_L$ are defined
in an analogous way using the equations defined by the products
$hC'_{y}$ instead of $C'_{y}h$.
The preorder $\LEQ_{LR}$ is generated by $\LEQ_{R}$ and $\LEQ_{L}$, and
$\sim_{LR}$ is the corresponding equivalence relation.
Using this notation, equation~\Bref{eqn:2a} can be rewritten
\begin{equation}
\label{eqn:3a}
C_{y}'h=\sum_{x\LEQ_{R} y}\alpha_{y,h,x}C_{x}'.
\end{equation}
The equivalence classes of  $\sim_L$, $\sim_R$ and $\sim_{LR}$ are
called \emph{left cells}, \emph{right cells} and
\emph{two-sided cells}, respectively.
If $x,y\in W$, we write $x<_{L} y$ to mean $x\LEQ_{L} y$ and
$x\not\sim_L y$; $<_{R}$ and $<_{LR}$ have similar meanings.
If $x\in W$ and $U,U'\subseteq W$, we will write $x\LEQ_{L} U'$ to mean
$x\LEQ_{L} y$ for all $y\in U'$ and $U\LEQ_{L} U'$ to mean $x\LEQ_{L} U'$
for all $x\in U$.
We extend the use of $<_{L}$, $\LEQ_{R}$, $<_{R}$, $\LEQ_{LR}$ and $<_{LR}$
in a similar fashion.

On several occasions throughout the paper we will be working with a Coxeter system $(W,S)$ while, inside it, we have fixed a parabolic subgroup $W'$
with corresponding Coxeter system $(W',S')$ for some $S'\subseteq S$.
It will be convenient at this point to fix some notation which we will follow throughout when we work in this context.
In particular we will write $\mathfrak X'=\mathfrak X_{S'}$ and $\mathfrak X^*=\{x^{-1}: x\in\mathfrak X'\}$,
so $\mathfrak X^*$ is the set of distinguished left coset representatives of $W'$ in $W$.
We will also denote by $\mathcal H'$ the Hecke algebra corresponding to the Coxeter system $(W',S')$.
The Kazhdan-Lusztig $C$-basis of $\mathcal H'$ is obtained by
restricting the Kazhdan-Lusztig $C$-basis of $\mathcal H$ to those elements
indexed by $W'$ (see \cite[\S~2.7]{Gec05}).
Let $\LEQ_L'$, $\LEQ_R'$, $\LEQ_{LR}'$, $\sim_L'$, $\sim_R'$, and
$\sim_{LR}'$ denote the Kazhdan-Lusztig preorders and equivalence
relations of the Coxeter system $(W',S')$ corresponding to those
for $(W,S)$ mentioned above.
Hence, if $u,v\in W'$ then clearly $u\LEQ_R' v$ implies $u\LEQ_R v$,
$u\sim_R' v$ implies $u\sim_R v$, and $u<_R' v$ implies $u<_R v$.
In Geck~\cite{Gec03} it was shown that the reverse implications also hold, so $\LEQ_R'$, $<_R'$ and $\sim_R'$ are just the restrictions of $\LEQ_R$, $<_R$ and $\sim_R$ respectively to $W'$.

In this paper we will be making the following convention.
In the cases the discussion involves working with $W'$ and $W$ at the same time,
 statements such as $w_1\LEQ_R' w_2$ or $w_1\sim_R' w_2$
will imply that $w_1,w_2\in W'$ and are not in $W\backslash W'$.
Also for $\mathfrak C$ a right cell of $W'$ the statement $w\LEQ_R'\mathfrak C$ will imply that $w\in W'$ and $w$ does not belong to $W\setminus W'$.

Comparing with equation~\eqref{eqn:3a}, now working inside $W'$, we get (for $y\in W'$ and $h\in\mathcal H'$)
$C'_yh=\sum_{x\LEQ_R'y}\alpha_{y,h,x}'C_x'$.
It follows from the basis property that $\alpha_{y,h,x}=\alpha_{y,h,x}'$
(the $\alpha_{y,h,x}$ as in equation~\eqref{eqn:3a}) whenever $y,x\in W'$ and $h\in\mathcal H'$.
Moreover, again with  $y,x\in W'$ and $h\in\mathcal H'$, we see that the nonzero $\alpha_{y,h,x}$ in equation~\eqref{eqn:3a} are
necessarily ones for which $x\LEQ_R'y$, so in this special case equation~\eqref{eqn:3a} becomes
\begin{equation}\label{eqn:3'}
C'_yh=\sum_{x\LEQ_R'y}\alpha_{y,h,x}C_x'
\end{equation}
One consequence of this is that when we restrict the relation ``$\to_s$'' from $(W,S)$ to $(W',S')$
(that is, for $x,y\in W'$ and $s\in S'$)
we see that this restriction is precisely the corresponding relation for the system $(W',S')$.

The following result will be useful in what follows.
For a proof see \cite[Corollary~1.9(c)]{Lus87}.
See also \cite[Lemma~5.3]{Gec05} for a more elementary algebraic proof in the case that $W$ is the symmetric group.

\begin{result}[{\cite[Corollary~1.9(c)]{Lus87}}]
\label{res:2c}
If $W$ is a crystallographic group and $x,y\in W$ are such that
$x\sim_{LR}y$ and $x\LEQ_R y$ then $x\sim_R y$.
\end{result}

For $w\in W$, let $M_w$ and $\hat{M}_w$ denote the
$\mathcal{H}$-modules with $A$-bases $\{C_y\Colon y\in W$ and $y\LEQ_{R} w\}$ and
$\{C_y\Colon y\in W$ and $y<_{R} w\}$, respectively, and let $S_w=M_w/ \hat{M}_w$.
Then $S_w$ is a Kazhdan-Lusztig cell module and affords the cell
representation corresponding to the right cell containing $w$.
If $\mathfrak{C}$ denotes the right cell of $W$ containing $w$, we also write
$S_{\mathfrak{C}}$ for $S_w$, $M_{\mathfrak{C}}$ for $M_w$ and
$\hat{M}_{\mathfrak{C}}$ for $\hat{M}_w$.

It will be convenient on occasion
to extend the scalars of the algebras under consideration.
Let $R$ be any commutative ring with 1 and let $A\rightarrow R$ be a
ring homomorphism.
With each $A$-module $L$, we have
an associated $R$-module $R\otimes_A L$,
which we will denote briefly as $L_{R}$.
In particular, we obtain an $R$-algebra $\mathcal{H}_{R}$,
and $\mathcal{H}_{R}$-modules $M_{R,w}=R\otimes M_w$,
$\hat{M}_{R,w}=R\otimes M_w$,
and Kazhdan-Lusztig cell modules
$S_{R,w}=M_{R,w}/ \hat{M}_{R,w}\cong R\otimes S_w$.
\\
In particular, we will use $F$ to denote any field containing the
field of fractions $\mathbf{Q}\big(q^{\frac{1}{2}}\big)$ of $A$,
and assume that the homomorphism $A\rightarrow F$ is inclusion.

\subsection{Specht modules}

If $N$ is a $\mathcal H$-module, a \emph{filtration} of $N$ is an
ascending sequence of $\mathcal H$-submodules,
$0=N_0\subseteq N_1\subseteq\ldots\subseteq N_r=N$.
In the case $W$ is the symmetric group, we call this filtration a \emph{Specht filtration} if the factors
$N_{i}/N_{i-1}$ are isomorphic to Specht $\mathcal H$-modules.
As we will be working with Specht filtrations, below we recall
the generalizations of the notions of diagram and tableau
commonly used, a description of which was given in \cite{MPa15}, as well as some basic facts about Specht modules.

A composition $\lambda$ of $m$ is a sequence $(\lambda_i)=(\lambda_1,\lambda_2,\ldots)$ of non-negative integers such that $\sum\lambda_i=m$.
We will use the notation $\lambda\vDash m$ (resp., $\lambda\vdash m$) to denote that $\lambda$ is a composition (resp., partition) of $m$.
We will also be using the following convention:
By $\lambda=(\lambda_1,\ldots,\lambda_r)\vDash m$ we will mean that $\lambda$ is the composition $(\lambda_i)$ of $m$, where $\lambda_i>0$ for $1\le i\le r$ and $\lambda_i=0$ for $i>r$ (and say that the composition $\lambda$ has $r$ parts).

For the rest of this subsection $W$ will denote the symmetric group $S_m$ on $\{1,\ldots,m\}$.
For $1\le i\le m-1$ let $s_i$ be the basic transposition $(i,\, i+1)$, and let $S=\{s_1,\ldots,s_{m-1}\}$.
Then $S$ is a system of Coxeter generators for $W$.
For $\lambda=(\lambda_1, \ldots , \lambda_r)\vDash m$ we define the subset $J(\lambda)$ of $S$ to be
$S\backslash \{s_{\lambda_1},s_{\lambda_1+\lambda_2},\ldots, s_{\lambda_1+\ldots+\lambda_{r-1}}\}$.

A \emph{diagram} $D$ is a finite subset $\mathbb{Z}^2$.
We will assume, where possible, that $D$ has no empty rows or columns.
These are the principal diagrams of \cite{MPa15}.
We will also assume that both rows and columns of $D$ are indexed
consecutively from 1.
The \emph{row-composition} $\lambda_D$ (respectively,
\emph{column-composition} $\mu_D$) of $D$ is defined by
$\lambda_{D,k}$ (respectively, $\mu_{D,k}$) is the number of nodes on
the $k$-th row (respectively, column) of $D$.
If $\lambda$ and $\mu$ are compositions,
we will write $\mathcal{D}^{(\lambda,\mu)}$ for the set
of principle diagrams $D$ with $\lambda(D)=\lambda$ and $\mu(D)=\mu$.
A \emph{special diagram} is a diagram obtained from a Young diagram
by permuting the rows and columns.
%

Since it is immediate that $\mathcal{D}^{(\lambda,\lambda')}$ consists
of a single diagram if $\lambda$ is a partition
(this is the Young diagram of shape $\lambda$ corresponding to partition $\lambda$),
it follows easily that $\mathcal{D}^{(\lambda,\mu)}$ consists of a single special
diagram if $\lambda$ and $\mu$ are compositions with
$\lambda''=\mu'$.

If $D$ is a diagram with $m$ nodes,
a \emph{$D$-tableau} is a bijection
$t\Colon D\rightarrow \{1,\ldots,m\}$ and
we refer to $t(i,j)$, where $(i,j)\in D$, as the
$(i,j)$-\emph{entry} of $t$.
The group $W=S_m$ acts on the set of $D$-tableaux in the obvious
way --- if $w\in W$,
an entry $i$ is replaced by $iw$ and $tw$ denotes the tableau
resulting from the action of $w$ on the tableau $t$.
We denote by $t^{D}$ and  $t_{D}$ the two $D$-tableaux obtained by
filling the nodes of $D$ with $1,\ldots,m$ by rows and by columns,
respectively, and we write $w_{D}$ for the element of $W$ defined by
$t^{D}w_{D}=t_{D}$.

Let $\lambda$ and $\mu$ be compositions of $m$ satisfying
$\mu'=\lambda''$, that is $\mu$ is a rearrangement of $\lambda'$.
Let $D$ and $D'$ be the unique diagrams in
$\mathcal{D}^{(\lambda,\mu)}$ and $\mathcal{D}^{(\lambda,\lambda')}$,
respectively.
Following \cite{DJa86}, the corresponding Specht module for $\mathcal H$ is
$S^{\lambda}=x_{\lambda}T_{w_{D'}}y_{\lambda'}\mathcal H$ where
$x_{\lambda}=\sum_{w\in W_{J(\lambda)}} T_w$ and
$y_{\lambda}=\sum_{w\in W_{J(\lambda)}} (-q)^{-l(w)}T_w$.
In the notation of \cite[Section~3]{MPa05},
$x_{\lambda}=q^{(1/2)l(w_{J(\lambda)})}C'_{w_{J(\lambda)}}
\mbox{ and }
y_{\lambda}=\left(-q^{-1/2}\right)^{l(w_{J(\lambda)})}
C_{w_{J(\lambda)}}$
From \cite[Corollary~4.3(ii)]{MPa15}, we see that
$S^{\lambda}=x_{\lambda}\mathcal Hy_{\lambda'}\mathcal H$, the product
$x_{\lambda}T_{w_{D}}y_{\mu}\neq0$
and
$x_{\lambda}T_{w_{D}}y_{\mu}\mathcal H$
$=x_{\lambda}\mathcal Hy_{\mu}\mathcal H$.
As in the proof of \cite[Lemma~4.3]{DJa86},
there is 
an element $d\in\mathfrak X_{J(\mu)}\cap \mathfrak X_{J(\lambda')}^{-1}$
with $d^{-1} W_{J(\mu)} d = W_{J(\lambda')}$ and,
consequently,
$T_d^{-1} y_{\mu} T_d = y_{\lambda'}$.
So, $x_{\lambda}T_{w_D}y_{\mu}\mathcal H$
$=x_{\lambda}\mathcal Hy_{\mu}\mathcal H$
$=x_{\lambda}\mathcal Hy_{\lambda'}\mathcal H$
$=S^{\lambda}$.

The following is an immediate consequence of certain results in~\cite{MPa05} and~\cite{MPa15}.

\begin{result}\label{Result:JPAATh3.5}
Let $\lambda,\mu\vDash m$ with $\lambda''=\mu'$.
Also let $D$ be the unique diagram in $\mathcal{D}^{(\lambda,\mu)}$.
Define $\theta:M_{w_{J(\mu)}}\to S^\lambda\,(=x_\lambda T_{w_D}C_{w_{J(\mu)}}\mathcal H)$ by $m\mapsto x_\lambda T_{w_D}m$ ($m\in M_{w_{J(\mu)}}$).
Then\\
(i) $\theta$ is a surjective $\mathcal H$-module homomorphism with $\ker\theta=\hat M_{w_{J(\mu)}}$ (so $\theta$ induces a natural $\mathcal H$-module isomorphism between $S_{w_{J(\mu)}}$ and $S^\lambda$).
\\
(ii) The set $\{x_\lambda T_{w_D}C_w:\ w\sim_R w_{J(\mu)}\}$ is an $A$-basis for $S^\lambda$.
\end{result}
\begin{proof}
Clearly $\theta$ is a well-defined surjective $\mathcal H$-module homomorphism since $M_{w_{J(\mu)}}=C_{w_{J(\mu)}}\mathcal H$.
We know from the proof of~\cite[Theorem 4.6]{MPa15} that $\{C_w:\ w\in W$ and $w<_R w_{J(\mu)}\}\subseteq\ker\theta$.
It follows that $\hat M_{w_{J(\mu)}}\subseteq\ker\theta$.
Hence $\{x_\lambda T_{w_D}C_w:\ w\sim_R w_{J(\mu)}\}$ is an $A$-spanning set for $S^\lambda$.
Again from~\cite[Theorem 4.6]{MPa15} we know that the set $\{x_\lambda T_{w_D}C_w:\ w\sim_R w_{J(\mu)}\}$ is $F$-linearly independent and hence $A$-linearly independent.
This completes the proof of (ii).
To complete the proof of (i) we argue as in the proof of~\cite[Theorem 3.5]{MPa05}.
So we let $r\in\ker\theta$.
Then $r=m+\sum_{w\sim_Rw_{J(\mu)}}\alpha_wC_w$ for some $m\in\hat M_{w_{J(\mu)}}$ and $\alpha_w\in A$.
We can deduce that $\sum_{w\sim_Rw_{J(\mu)}}\alpha_w(C_w\theta)=0$ forcing all $\alpha_w$ in the previous sums to be equal to zero since, as we have seen, $\{C_w\theta:\ w\sim_R w_{J(\mu)}\}$ is $A$-linearly independent.
We conclude that $r=m\in\hat M_{w_{J(\mu)}}$ and hence $\ker\theta=\hat M_{w_{J(\mu)}}$.
\end{proof}

\begin{remark}\label{RemarkSpecht}
In view of Result~\ref{Result:JPAATh3.5}, and keeping the notation introduced there, we can see that $S_{w_{J(\mu)}}\cong S^\nu$, where $\nu$ is the partition of $m$ satisfying $\nu=\lambda''$ (since $\nu''=\nu=\lambda''=\mu'$).
From~\cite[\S5]{KLu79} (see also \cite[Corollary 5.8]{Gec05}), we also know that if $\mathfrak{C}$ is a (right) cell of $W=S_m$ with $\shape \mathfrak{C}=\nu\,(=\mu')$ then $S_{\mathfrak{C}}\cong S_{w_{J(\mu)}}$.
We can deduce that $S_{\mathfrak{C}}\cong S^\nu$ whenever $\mathfrak{C}$ is a right cell of $S_m$ with $\shape \mathfrak{C}=\nu$.

\end{remark}

\subsection{Kazhdan-Lusztig cells in $S_n$: induction and restriction}\label{Subsec:KLCellsInSn}


If $W$ is a Coxeter group, each cell of $W$ provides an integral representation of $W$; see
Kazhdan and Lusztig \cite[\S~1]{KLu79}. Barbasch and Vogan \cite{BVo83}
have addressed the question of induction and restriction of such
representations in relation to parabolic subgroups, where $W$ is a
Weyl group.
They showed that Kazhdan-Lusztig cells are compatible with parabolic subgroups so that the $C$-basis is compatible with induction and restriction of representations.
These results have been generalized to all Coxeter groups by
Roichman~\cite[Theorem~5.2]{Roi98} and Geck~\cite[Theorem 1]{Gec03}.

In the case of the symmetric group $S_m$, the Robinson-Schensted
correspondence gives a combinatorial method of identifying the
Kazhdan-Lusztig cells.
The Robinson-Schensted correspondence is a bijection of $S_m$
to the set of pairs of standard Young tableaux
$(\mathcal{P},\mathcal{Q})$ of the same shape and with $m$ entries,
where the shape of a tableau is the partition counting the number of
entries on each row.
See Fulton \cite{Ful97} or Sagan \cite{Sagan}
for a good description of this correspondence.
Denote this correspondence by
$w\mapsto(\mathcal{P}(w),\mathcal{Q}(w))$.
Then $\mathcal{Q}(w)=\mathcal{P}(w^{-1})$.
The \emph{shape} of $w$, denoted by $\shape{w}$, is defined to be
the common shape of the Young tableaux $\mathcal{P}(w)$ and
$\mathcal{Q}(w)$.
The tableaux $\mathcal{P}(w)$ and $\mathcal{Q}(w)$ are called
the \emph{insertion tableau} and the \emph{recording tableau},
respectively, for $w$.
\par
We will use $\unlhd$ to denote \emph{dominance} of partitions
(see \cite[p.26]{Ful97}) and use $\lhd$ to denote strict dominance.
\par
The following result characterises the cells in $S_m$;
see~\cite{KLu79} and also \cite{Ari00} or \cite{Gec05} for more detailed proofs.
%
\begin{result}
[{\cite{KLu79}, see also \cite[Theorem A]{Ari00} or \cite[Corollary 5.6]{Gec05}}]
\label{res:5c}
If $\mathcal{P}$ is a fixed standard Young tableau then
the set $\{w\in W\Colon \mathcal{P}(w)=\mathcal{P}\}$
is a left cell of $W$ and the set
$\{w\in W\Colon\mathcal{Q}(w)=\mathcal{P}\}$ is a right cell of $S_m$.
Conversely, every left cell and every right cell arises in this way.
Moreover, the two-sided cells are the subsets of $W$ of the form
$\{w\in W\Colon\shape\mathcal{P}(w)\mbox{ is a fixed partition.}\}$
\end{result}


The \emph{shape} $\shape{\mathfrak{C}}$ of a cell $\mathfrak{C}$ is
$\shape{w}$ for any $w\in\mathfrak{C}$.
\par
For the rest of this section,
$W'$ and $W$  will be the symmetric groups $S_{n}$ on $\{1,\dots,n\}$
and $S_{n+1}$ on $\{1,\ldots,n+1\}$, respectively, with the natural
embedding.
Let $s_i=(i,i+1)$ for $1\le i\le n$,
let $S'=\{s_1,\ldots,s_{n-1}\}$ and $S=\{s_1,\ldots,s_n\}$, so that
$(W',S')$ and $(W,S)$ are Coxeter systems.
Recalling the notation $\mathfrak{X}'=\mathfrak{X}_{S'}$,
we have $\mathfrak{X}'=\{x_i\colon 1\le i\le n+1\}$, where
$x_i=(i,i+1,\dots,n,n+1)=s_n\cdots s_i$ (the empty product is $1$ by convention).

For a Young diagram $D$ corresponding to the partition
$\lambda=(\lambda_1,\dots,\lambda_r)$, let $\ic(D)$ and $\oc(D)$ be the
sets of \emph{inner corners} and \emph{outer corners}, respectively,
of $D$; that is,
\[
\begin{array}{rcl}
\ic(D) & = &
\{(i,\lambda_i))\colon 1\le i\le r-1
\mbox{ where } \lambda_{i}>\lambda_{i+1}\}
\cup
\{(r,\lambda_r)\},
\\
\oc(D) & = &
\{(1,\lambda_1+1)\}
\cup
\{(i,\lambda_i+1))\colon
2\le i\le r \mbox{ where } \lambda_{i-1}>\lambda_i\}
\cup
\{(r+1,1)\}.
\end{array}
\]
We denote by $<$ the total order on the nodes of a diagram given by
$(i,j)<(i',j')$ if, and only if, $i<i'$ or $i=i'$ and $j<j'$.

Finally for this subsection we include some explicit results of the Barbasch and Vogan induction--restriction theorems for Kazhdan--Lusztig cells in the symmetric group as they were described in~\cite{MPa17}.

\begin{result}[{\cite[Proposition 2.1]{MPa17}}]
\label{prop:2.1c}
Let $\mathfrak{C}$ be a right cell of $W$,
let $A$ be the recording tableau of elements of $\mathfrak{C}$ and
let $D$ be its underlying diagram.
For each $k\in \oc(D)$, let $A_k$ be the tableau obtained from $A$ by
adding the entry $n+1$ at node $k$ and let $\mathfrak{C}_k$ be the
right cell of $W'$ corresponding to the recording tableau $A_k$.
Then $\mathfrak{C}\mathfrak{X}'=\bigcup_{k\in\oc(D)}\mathfrak{C}_k$.
\par
Furthermore, if $k,k'\in\oc(D)$ and $k<k'$ then
$\shape{\mathfrak{C}_{k'}}\lhd\shape{\mathfrak{C}_{k}}$.
\end{result}
\begin{result}[{\cite[Proposition 2.2]{MPa17}}]
\label{prop:2.2c}
Let $\mathfrak{C}$ be a right cell of $W'$ and
let $A$ be the recording tableau of elements of $\mathfrak{C}$
and let $D$ be its underlying diagram.
For each $k\in \ic(D)$, if $i(k)$ is the entry on the first row of $A$
removed by reverse inserting from node $k$ and $A'$ is the resulting
tableau, let $d_k=x_{i(k)}^{-1}$ and let $A_k=A'd_k$
(so that $A_k$ is a standard Young tableau on $1,\dots,n$).
Let $\mathfrak{C}_k$ be the right cell of $W$ corresponding to the
recording tableau $A_k$.
Then $\mathfrak{C}=\bigcup_{k\in \ic(D)}d_k\mathfrak{C}_k$.
\par
Furthermore, if $k,k'\in\ic(D)$ and $k<k'$ then
$\shape{\mathfrak{C}_{k}}\lhd\shape{\mathfrak{C}_{k'}}$
and $d_{k}\LEQ d_{k'}$.
\end{result}

\section{Specht filtrations}\label{Section:SpechtFiltrations}

\subsection{Kazhdan-Lusztig cell modules: induction and restriction}\label{Subsec:KLCellsModulesIndRestr}

We continue with $W$ a Coxeter group having $W'$ as a parabolic subgroup.
Recall our notation that  $\mathfrak{X}'$ (resp., $\mathfrak X^*$) is the set of
distinguished right (resp., left) coset representatives of $W'$ in $W$ and
$\mathcal{H}'$ (resp., $\mathcal{H}$) is the Hecke algebra of $W'$ (resp., $W$) over the ring $A$.
Then $\mathcal{H}'$ is a subalgebra of $\mathcal{H}$, and
$\mathcal{H}$ is a free left $\mathcal{H}'$-module with free
$\mathcal{H}'$-basis $\{T_w\colon w\in\mathfrak{X}'\}$ and, hence,
$\mathcal{H}$ is a flat left $\mathcal{H}'$-module
(see Lam~\cite[Proposition~4.3]{lam99}).
For $E$  a subset of $\mathcal H'$  we denote by $E\mathcal H'$ the right ideal of $\mathcal H'$ generated by $E$.
(Similarly for $E\mathcal H$ when $E\subseteq\mathcal H$.)
One can then easily observe that if $M$ is an ideal of $\mathcal{H}'$, then $M\mathcal{H}$ is an ideal
of $\mathcal{H}$, and
$M\mathcal{H}\cong M\otimes_{\mathcal{H}'}\mathcal{H}$.

In what follows, for $\mathfrak C'$ a right cell in $W'$ we denote by $M_{\mathfrak C'}$ (resp., $\hat M_{\mathfrak C'}$)
the right ideal $\langle C_w: w\LEQ_R'\mathfrak C'\rangle_A$ (resp., $\langle C_w: w<_R'\mathfrak C'\rangle_A$) of $\mathcal H'$
(recall our conventions for $\LEQ_R'$, $<_R'$ from Subsection~\ref{Subsection:HeckeABackgr}).
Similarly for $\mathfrak C$ a right cell in $W$ we will denote by $M_{\mathfrak C}$ (resp., $\hat M_{\mathfrak C}$)
the right ideal $\langle C_w: w\LEQ_R\mathfrak C\rangle_A$
(resp., $\langle C_w: w<_R\mathfrak C\rangle_A$)  of $\mathcal H$.
The corresponding cell modules are denoted by $S_{\mathfrak C'}$ and $S_{\mathfrak C}$ respectively.

Now let $\mathfrak{C}'$ be a right cell in $W'$.
As $M_{\mathfrak{C}'}$ and $\hat{M}_{\mathfrak{C}'}$ are right
$\mathcal{H}'$-ideals,
$M_{\mathfrak{C}'}\mathcal{H}$ and
$\hat{M}_{\mathfrak{C}'}\mathcal{H}$ are right $\mathcal{H}$-ideals,
and hence
$M_{\mathfrak{C}'}\mathcal{H}/\hat{M}_{\mathfrak{C}'}\mathcal{H}\cong
(M_{\mathfrak{C}'}\otimes_{\mathcal H'}\mathcal H)
/ (\hat M_{\mathfrak{C}'}\otimes_{\mathcal H'}\mathcal H)
\cong
(M_{\mathfrak{C}'}/\hat M_{\mathfrak{C}'})\otimes_{\mathcal H'}\mathcal H
\cong
S_{\mathfrak{C}'}\otimes_{\mathcal{H}'}\mathcal{H}$
as right $\mathcal{H}$-modules,
since $\mathcal{H}$ is a flat left $\mathcal{H}'$-module.

In this section, we investigate the construction of filtrations by cell modules for the
$\mathcal H$-module $S_{\mathfrak{C}'}\uparrow\mathcal H$ and
the $\mathcal H'$-module $S_{\mathfrak{C}}\downarrow\mathcal H'$
where $\mathfrak{C}'$ is a cell of $W'$ and $\mathfrak{C}$ is a cell
of~$W$ in the case $W$ is the symmetric group.
In view of Remark~\ref{RemarkSpecht} such filtrations are also Specht filtrations.

For the rest of the paper, unless explicitly mentioned otherwise,
$W'$ and $W$  will be the symmetric groups $S_{n}$ on $\{1,\dots,n\}$
and $S_{n+1}$ on $\{1,\ldots,n+1\}$, respectively, with the natural
embedding.

\begin{proposition}
\label{prop:4.1c}
Let $\mathfrak{C}$ be a right cell of $W'$.
Let $\lambda=\shape\mathfrak{C}$, let $D$ be the unique diagram in $\mathcal{D}^{(\lambda,\lambda')}$
and write $\oc(D)=\{k_1,\ldots,k_p\}$ where $k_j<k_i$ if $i<j$.
Then the induced $\mathcal H$-module
$S_{\mathfrak{C}}\uparrow\mathcal H$ has a filtration
$\{0\}=N_0\subseteq N_1\subseteq\ldots\subseteq N_p$ such that
$N_{i}/N_{i-1}\cong S_{\mathfrak{C}_i}$
(where, for simplicity, we write $\mathfrak{C}_i$ in the place of right cell $\mathfrak{C}_{k_i}$of $W$).
\end{proposition}
\begin{proof}
First note that the sets
$\{w\in W'\colon w\LEQ_R'\mathfrak{C}\}$ and
$\{w\in W'\colon  w<_R'\mathfrak{C}\}$ both satisfy the hypothesis of
Geck~\cite[Lemma~2.2 and Corollary~3.4]{Gec03}.
Hence,
$M_{\mathfrak{C}}\mathcal H$
$=\langle C_{yv}\colon y\LEQ_R'\mathfrak{C}, v\in\mathfrak X'\rangle_A$
$=\langle C_yT_v\colon y\LEQ_R'\mathfrak{C}, v\in\mathfrak X'\rangle_A$
and $\hat M_{\mathfrak{C}'}\mathcal H$
$=\langle C_{yv}\colon y<_R'\mathfrak{C}, v\in\mathfrak X'\rangle_A$
$=\langle C_yT_v\colon y<_R'\mathfrak{C}, v\in\mathfrak X'\rangle_A$.
We write
$\mathfrak D=\{w\in W\colon w=yv$ for some
$v\in\mathfrak X'$ and $y\in W'$ with $y<_R'\mathfrak{C}\}$.
So, $\hat M_{\mathfrak{C}}\mathcal H$
$=\langle C_{w}\colon w\in\mathfrak{D}\rangle_A$.
Let
$L_0=\hat M_{\mathfrak{C}}\mathcal H$.

From Result~\ref{prop:2.1c}, we see that
Then $\mathfrak{C}\mathfrak{X}'=\bigcup_{k\in\oc(D)}\mathfrak{C}_k$,
where $\oc(D)$ is the set of outer corners of $D$ and, for each
$k\in\oc(D)$, $\mathfrak{C}_k$ is a right cell of $W$.
Moreover, we can write  $\oc(D)=\{k_1,\ldots,k_p\}$ where $k_j<k_i$
if $i<j$. So,
$\shape{\mathfrak{C}_{k_i}}\lhd\shape{\mathfrak{C}_{k_j}}$ if $i<j$.
From \cite[Theorem~5.1]{Gec05}, we see that
$\mathfrak{C}_{k_i}<_{LR}\mathfrak{C}_{k_j}$ if, and only if, $i<j$.
It follows that if $\mathfrak{C}_{k_i}<_{R}\mathfrak{C}_{k_j}$
then $i<j$.

For $1\le j\le p$,
let $L_j=\langle C_w\colon w\in\mathfrak{D}\cup
\bigcup_{1\le i\le j}\mathfrak{C}_{k_i}\rangle_A$.
Suppose that $w'\in W$ and $w'\LEQ_{R}\mathfrak{C}_{k_j}$;
that is, $w'\LEQ_{R}w$ for some $w\in\mathfrak{C}_{k_j}$.
If $w'\in\mathfrak{C}\mathfrak{X}'$, then
$w'\sim_{R}\mathfrak{C}_{k_i}$ for some $i$ with $1\le i\le j$.
If $w'\notin\mathfrak{C}\mathfrak{X}'$, then
$w'=y'v'$ where $y'\in W'$, $y'\notin\mathfrak{C}$, and
$v'\in\mathfrak{X}'$.
Writing $w=yv$ where $y\in W'$ and $v\in\mathfrak{X}'$,
by the `right' version of \cite[\S4($\dagger$)]{Gec03}, $y'\LEQ_R' y$.
As $y'\not\sim_R' y$, $y'<_R'\mathfrak{C}$ and $w'\in\mathfrak{D}$.
Since $i<j$ whenever $\mathfrak{C}_{k_i}<_R\mathfrak{C}_{k_j}$ we have thus shown that $w'\LEQ_{R}\mathfrak{C}_{k_j}$ implies $w'\in L_j$.
In view of the fact that $L_i\subseteq L_j$ for $i\le j$ we get that $L_j$ is an $\mathcal H$-module containing
$M_{\mathfrak{C}_{k_j}}$,
$L_{j-1}$ is an $\mathcal H$-module containing
$\hat M_{\mathfrak{C}_{k_j}}$,
and
$L_j/L_{j-1}\cong
M_{\mathfrak{C}_{k_j}}/\hat M_{\mathfrak{C}_{k_j}}
=S_{\mathfrak{C}_{k_j}}$.

By construction, $L_p=M_{\mathfrak{C}}\mathcal H$.
Hence, $L_p/L_0=M_{\mathfrak{C}}\mathcal{H}/\hat{M}_{\mathfrak{C}}\mathcal{H}
\cong
(M_{\mathfrak{C}}/\hat M_{\mathfrak{C}})\otimes_{\mathcal H'}\mathcal H
\cong S_{\mathfrak{C}}\otimes_{\mathcal H'}\mathcal H$
$=S_{\mathfrak{C}}\uparrow\mathcal H$.
\par
To complete the proof, we let $N_j=L_j/L_0$ for $0\le j\le p$.
\end{proof}

Before establishing the corresponding result for the restricted module
of a cell module, we prove the following technical lemma.
Note that results~\ref{lem:4.2a} and \ref{cor:4.3a}  below are true for $W$
an arbitrary Coxeter group with $W'$ a parabolic subgroup of~$W$.
It would be useful at this point to recall equations~\eqref{eqn:3a} and~\eqref{eqn:3'} and the relevant notation introduced in Subsection~\ref{Subsection:HeckeABackgr}.


\begin{lemma}[{compare~\cite[Proposition 3.3]{Gec03}, \cite[Corollary 3.5]{Gec06b}}]
\label{lem:4.2a}
Let $W$ be a Coxeter group and $W'$ a parabolic subgroup of~$W$.
Also let  $y\in\mathfrak{X}^*$, $v\in W'$ and $h\in\mathcal{H}'$.
Then
\begin{equation}
\label{eqn:4a}
C_{yv}'h=\sum_{u\LEQ_{R}' v}\alpha_{v,h,u}C_{yu}'
+\sum_{x<y,\ u\LEQ_{LR}'v,\ xu\LEQ_{R}yv}\beta_{yv,h,xu}C_{xu}'
\end{equation}
where $\beta_{yv,h,xu}\in A$ and $\alpha_{v,h,u}$ is as in
equations~\Bref{eqn:3a} and~\Bref{eqn:3'}.
\end{lemma}

\begin{proof}
The usual $C'$-basis of $\mathcal{H}$ can be written as
$\{C_{xu}'\colon x\in\mathfrak{X}^*, u\in W'\}$.
Comparing with~\eqref{eqn:3a},
since $y\in\mathfrak{X}^*$, $v\in W'$ and $h\in\mathcal{H}'$,
by hypothesis,
\begin{equation}
\label{eqn:5a}
C_{yv}'h=\sum_{x\in\mathfrak{X}^*,\ u\in W',\ xu\LEQ_{R}yv}
\alpha_{yv,h,xu}C_{xu}'
\end{equation}
where $\alpha_{yv,h,xu}\in A$.
Moreover, $\{\tilde T_xC_u'\colon x\in\mathfrak{X}^*, u\in W'\}$ is also
an $A$-basis for $\mathcal{H}$;
see~\cite[Section~3]{Gec03} and the discussion after
Result~\ref{res:1a}.
Geck shows in~\cite[Proposition~3.3]{Gec03} that, for
any $y\in\mathfrak{X}^*$ and $v\in W'$,
\begin{equation}
\label{eqn:6a}
C_{yv}'=\tilde T_{y}C_{v}'+
\sum_{x\in\mathfrak{X}^*,\ u\in W',\ x<y,\ u\LEQ_{L}' v}
\zeta_{yv,xu}\tilde T_{x}C_{u}'
\end{equation}
where $\zeta_{yv,xu}\in A$.
We can invert this system of equations to get
\begin{equation}
\label{eqn:7a}
\tilde T_{y}C_{v}'=C_{yv}'+
\sum_{x\in\mathfrak{X}^*,\ u\in W',\ x<y,\ u\LEQ_{L}' v}
\eta_{yv,xu}C_{xu}'
\end{equation}
where $\eta_{yv,xu}\in A$.

In the following computation, $x,x'\in\mathfrak{X}^*$ and
$u,u',u''\in W'$.
Taking $h\in\mathcal{H}'$, $y\in\mathfrak{X}^*$ and $v\in W'$
and using equation~\Bref{eqn:3a}, we get
\[
\begin{array}{rcl}
C_{yv}'h
&=&
\tilde T_{y}C_{v}'h+
\sum_{x<y,\ u\LEQ_{L}' v}\zeta_{yv,xu}\tilde T_{x}C_{u}'h
\\[10pt]
&=&
\tilde T_{y}
\left(\sum_{u\LEQ_{R}' v}\alpha_{v,h,u}C_{u}'\right)
+
\sum_{x<y,\ u\LEQ_{L}' v}
\zeta_{yv,xu}\tilde T_{x}
\left(\sum_{u'\LEQ_{R}' u}\alpha_{u,h,u'}C_{u'}'\right)
\\[10pt]
&=&
\sum_{u\LEQ_{R}' v}\alpha_{v,h,u}\tilde T_{y}C_{u}'
+
\sum_{x<y,\ u'\LEQ_{LR}' v}
\xi_{yv,h,xu'}\tilde T_{x}C_{u'}'
\end{array}
\]
where $\xi_{yv,h,xu'}\in A$.
Hence,
\[
\begin{array}{rcl}
C_{yv}'h
&=&
\sum_{u\LEQ_{R}' v}\alpha_{v,h,u}
\left(C_{yu}'+
\sum_{x'<y,\ u'\LEQ_{L}' u}
\eta_{yu,x'u'}C_{x'u'}'
\right)
\\
&&
+
\sum_{x<y,\ u'\LEQ_{LR}' v}
\xi_{yv,h,xu'}
\left(C_{xu'}'+
\sum_{x'<x,\ u''\LEQ_{L}' u'}
\eta_{xu',x'u''}C_{x'u''}'
\right)
\\
&=&
\sum_{u\LEQ_{R}' v}\alpha_{v,h,u}C_{yu}'
+
\sum_{x<y,\ u\LEQ_{LR}' v}
\beta_{yv,h,xu}C_{xu}'
\end{array}
\]
which, combined with~\Bref{eqn:5a}, gives the required result.
\end{proof}

\begin{corollary}
\label{cor:4.3a}
Under the hypothesis of Lemma~\ref{lem:4.2a} and with $\alpha_{v,h,u}$ and $\beta_{yv,h,xu}$ as defined therein, we have
\begin{eqnarray}
C_{v}h&=&\sum_{u\LEQ_{R}'v}\lambda_{v,h,u}C_{u}\hspace{2em}\mbox{and}
\label{eqn:8a}
\\
C_{yv}h&=&\sum_{u\LEQ_{R}' v}\lambda_{v,h,u}C_{yu}
+\sum_{x<y,\ u\LEQ_{LR}'v,\ xu\LEQ_{R}yv}\mu_{yv,h,xu}C_{xu}
\label{eqn:9a}
\end{eqnarray}
where
$\lambda_{v,h,u}=(-1)^{l(u)-l(v)}
\overline{\alpha_{v,h\jmath_{\mathcal{H}'},u}}\in A$,
$\mu_{yv,h,xu}=(-1)^{l(xu)-l(yv)}
\overline{\beta_{yv,h\jmath_{\mathcal{H}'},xu}}\in A$.
\end{corollary}

\begin{proof}
Let $h'=h\jmath_{\mathcal{H}'}\in\mathcal{H}'$ and apply
$\jmath_{\mathcal{H}'}$ to the equation
$C_{y}'h'=\sum_{x\LEQ_{R}' y}\alpha_{y,h',x}C_{x}'$ obtained from~\eqref{eqn:3a} by replacing $h$ by $h'$, to get
$(-1)^{l(y)}C_{y}h$
$=\sum_{x\LEQ_{R}' y}\overline{\alpha_{y,h',x}}(-1)^{l(x)}C_{x}$.
Equation~\eqref{eqn:8a} follows.
Similarly, applying $\jmath_{\mathcal{H}}$ to the equation
obtained from equation~\eqref{eqn:4a} by replacing $h$ with $h'$ and
noting that $\jmath_{\mathcal{H}'}$ is the automorphism of
$\mathcal{H}'$ obtained by restricting $\jmath_{\mathcal{H}}$
to $\mathcal{H}'$, we get
$(-1)^{l(yv)}C_{yv}h$
$=\sum_{u\LEQ_{R}' v}\overline{\alpha_{v,h,u}}(-1)^{l(yu)}C_{yu}
+\sum_{x<y,\ u\LEQ_{LR}'v,\ xu\LEQ_{R}yv}
\overline{\beta_{yv,h,xu}}(-1)^{l(xu)}C_{xu}$.
Equation~\eqref{eqn:9a} follows.
\end{proof}

%
%


\begin{proposition}
\label{prop:4.3c}
Let $\mathfrak{C}$ be a right cell of $W$.
Let $\lambda=\shape\mathfrak{C}$, let $D$ be the unique diagram in $\mathcal{D}^{(\lambda,\lambda')}$ and write
$\ic(D)=\{k_1,\ldots,k_p\}$ where $k_i<k_j$ if $i<j$.
Then the restricted $\mathcal H'$-module
$S_{\mathfrak{C}}\downarrow\mathcal H'$ has a filtration
$\{0\}=N_0\subseteq N_1\subseteq\ldots\subseteq N_p$ such that
$N_{i}/N_{i-1}\cong S_{\mathfrak{C}_i}$
(where, for simplicity, we write $\mathfrak{C}_i$ in place of $\mathfrak{C}_{k_i}$).
\end{proposition}

\begin{proof}
From Result~\ref{prop:2.2c}, we see that
$\mathfrak{C}=\bigcup_{k\in\ic(D)}d_k\mathfrak{C}_k$
where $\ic(D)$ is the set of inner corners of $D$,
$\mathfrak{C}_k$ is a right cell of $W'$ and $d_k\in\mathfrak{X}^{*}$
for $k\in\ic(D)$.
Moreover, writing  $\ic(D)=\{k_1,\ldots,k_p\}$ where $k_i<k_j$ if
$i<j$, we get
$\shape{\mathfrak{C}_{k_i}}\lhd\shape{\mathfrak{C}_{k_j}}$
and $d_{k_i}\LEQ d_{k_j}$ if $i<j$.
From \cite[Theorem~5.1]{Gec05}, we see that
$\mathfrak{C}_{k_i}<_{LR}'\mathfrak{C}_{k_j}$ if, and only if, $i<j$.
It follows from Result~\ref{res:2c} that if $\mathfrak{C}_{k_i}<_{R}'\mathfrak{C}_{k_j}$
then $i<j$.

We write $\mathfrak{D}=\{z\in W\colon z<_R\mathfrak{C}\}$.
Let $L_0=\langle C_z: z\in\mathfrak D\rangle_A$
 and for $1\le j\le p$ let $L_j$ be the $A$-submodule of $\mathcal H$
spanned by $\{C_w\colon w\in\mathfrak{D}\cup \bigcup_{1\le i\le j}d_{k_i}\mathfrak{C}_{k_i}\}$.
Indeed, since this spanning set is a subset of an $A$-basis of $\mathcal H$, it is an $A$-basis of $L_j$.
It is clear from this definition that $L_0\subseteq L_1\subseteq\ldots\subseteq L_p$.

For the rest of the proof, for the sake of simplicity, we will be writing $\mathfrak C_j$ in place of $\mathfrak C_{k_j}$ and $d_j$ in place of $d_{k_j}$ for $1\le j\le p$.

We will show that $L_j$ is an $\mathcal H'$-module, for $0\le j\le p$, and that
the cell module $S_{\mathfrak{C}_j}$ is $\mathcal H'$-isomorphic to $L_j/L_{j-1}$, for $1\le j\le p$.

First observe that $L_0=\hat M_{\mathfrak{C}}$ and $L_p=M_{\mathfrak{C}}$ as subsets of $\mathcal H$.
Hence, $L_0$ and $L_p$ are $\mathcal H'$-modules via the natural action of $\mathcal H'$ by right multiplication.
In fact, as $\mathcal H'$-modules $L_0=\hat M_{\mathfrak{C}}\downarrow\mathcal H'$, $L_p= M_{\mathfrak{C}}\downarrow\mathcal H'$
and $L_p/L_0$ is isomorphic to the restricted module
$(M_{\mathfrak{C}}/\hat M_{\mathfrak{C}})\downarrow\mathcal H'=S_{\mathfrak{C}}\downarrow\mathcal H'$.

Now fix $i$ with $1\le i\le p$ and let $v\in \mathfrak C_i$ and $h\in\mathcal H'$.
In view of Corollary~\ref{cor:4.3a} we get
\begin{equation}\label{eq:(6)}
C_{d_iv}h=\sum_{u\sim_R'v}\lambda_{v,h,u}C_{d_iu}+\sum_{u<_R'v}\lambda_{v,h,u}C_{d_iu}+\sum_{x<d_i,\ u\LEQ_{LR}'v,\ xu\LEQ_Rd_iv}\mu_{d_iv,h,xu}C_{xu}
\end{equation}
where $\lambda_{v,h,u},\mu_{d_iv,h,xu}\in A$ and $C_vh=\sum_{u\LEQ_R'v}\lambda_{v,h,u}C_u$.
Denote the first, second and third sum on the right-hand-side of~\eqref{eq:(6)} by $\Sigma_1$, $\Sigma_2$ and $\Sigma_3$ respectively
and let $w\in W$ be such that $C_w$ appears with non-zero coefficient in any one of these three sums.
It follows from this that $w\LEQ_Rd_iv$ and hence $w\LEQ_R\mathfrak C$. 
Suppose first that $C_w$ appears with non-zero coefficient in $\Sigma_1$.
This means that $w=d_iu$ for some $u\in\mathfrak C_i$, so $C_w\in L_i$.
Next, suppose that $C_w$ appears with non-zero coefficient in  either $\Sigma_2$ or $\Sigma_3$.
We consider the cases $w<_R\mathfrak C$ and $w\in\mathfrak C$ separately.
If  $w<_R\mathfrak C$, then $w\in\mathfrak D$, so $C_w\in L_0\subseteq L_i$.
On the other hand, if $w\in\mathfrak C$, we see from~\eqref{eq:(6)} that $w\in\cup_{1\le l<i}d_{l}\mathfrak C_{l}$, so again $C_w\in L_i$.
This shows that for $i\in\{1,\ldots,p\}$, $z\in d_{i}\mathfrak C_{i}$ implies $C_zh\in L_i$.
Invoking the facts that $L_0$ is an $\mathcal H'$-module and that $L_i\subseteq L_j$ whenever $i\le j$, we conclude
that $L_j$ is an $\mathcal H'$-module for $0\le j\le p$.

Finally, fix $j$ with $1\le j\le p$.
To establish an $\mathcal H'$-isomorphism between  $S_{\mathfrak{C}_j}$ and $L_j/L_{j-1}$,
define $\theta\colon M_{\mathfrak{C_j}}\to L_j/L_{j-1}$ by
$\left(\sum_{u\LEQ_{R}'\mathfrak{C}_j}\xi_uC_u\right)\theta=
\Big(\sum_{u\in\mathfrak{C}_j}\xi_uC_{d_ju}\Big)+L_{j-1}$
where $\xi_u\in A$.
In view of~\eqref{eq:(6)}, also comparing with equations~\eqref{eqn:8a} and~\eqref{eqn:9a} in Corollary~\ref{cor:4.3a}, it is easy to check that $\theta$ is an $\mathcal{H}'$-module homomorphism with $\ker\theta=\hat M_{\mathfrak{C}_j}$.
(In fact, choosing the $A$-bases
$\{C_{u}+\hat M_{\mathfrak{C}_j}\colon u\in\mathfrak{C}_j\}$ and
$\{C_{d_ju}+L_{j-1}\colon u\in\mathfrak{C}_j\}$ for
$M_{\mathfrak{C}_j}/\hat M_{\mathfrak{C}_j}$ and $L_j/L_{j-1}$,
respectively, we see that these two $\mathcal{H}'$-modules affords
the same matrix representation.)

To complete the proof, we let $N_i=L_i/L_0$ for $0\le i\le p$.
\end{proof}

\subsection{Specht modules: induction and restriction}\label{Subsec:SpechtModulesIndRestr}

In this subsection we make use of the results in Subsection~\ref{Subsec:KLCellsModulesIndRestr} in order to obtain Specht filtrations (over the ring $A$) for the induced and the restricted Specht modules via $C$-bases for these modules.

\begin{theorem}\label{thm:SpFilForS2H}
Let $\lambda,\mu\vDash n$ with $\lambda''=\mu'$, let $\mathfrak C$ be the right cell of $W'$ containing $w_{J(\mu)}$ and let $D$ be the underlying diagram of the recording tableau of the elements of $\mathfrak C$.
Also let $E$ be the unique diagram in $\mathcal D^{(\lambda,\mu)}$.
Write ${\rm oc}(D)=\{k_1,\ldots,k_p\}$, where $k_i<k_j$ if $i>j$, and for $j=1,\ldots,p$ let $\mathfrak C_j=\mathfrak C_{k_j}$ where $\mathfrak C_{k_j}$ is the cell of $W$ corresponding to node $k_j$ as in the statement of Result~\ref{prop:2.1c}
(so that $\mathfrak C\mathfrak X'=\cup_{1\le j\le p}\mathfrak C_j$).
Finally let $\overline{L}_i=\langle x_{\lambda}T_{w_E}C_w:\ w\in\cup_{1\le j\le i}\mathfrak C_j\rangle_A$ for $i=1,\ldots,p$.
Then, $\{0\}=\overline L_0\subseteq\overline L_1\subseteq\ldots\subseteq\overline L_p=S^\lambda\mathcal H$ is a Specht filtration for $S^\lambda \mathcal H$ with $\overline L_i/\overline L_{i-1}\cong S^{\nu^{(i)}}$, where $\nu^{(i)}=\shape\mathfrak C_i$ for $1\le i\le p$.
(In particular, $\nu^{(i)}\lhd\nu^{(j)}$ if $i<j$.)
\end{theorem}
\begin{proof}
We assume the hypothesis.
We know from Result~\ref{prop:2.1c} that $\cup_{1\le j\le p}\mathfrak C_j=\{ux:$ $u\in\mathfrak C$ and $x\in\mathfrak X'\}$.
First we show that $S^\lambda \mathcal H$ is free as an $A$-module having the set $\mathcal B=\{x_{\lambda}T_{w_E}C_w: w\in\cup_{1\le j\le p} \mathfrak{C_j}\}=\{x_{\lambda}T_{w_E}C_w: w\in\mathfrak C\mathfrak X'\}$ as an $A$-basis.
For this, set $\theta:M_{\mathfrak C}\mathcal H\to S^{\lambda}\mathcal H:$ $m\mapsto x_{\lambda}T_{w_E}m$ ($m\in M_{\mathfrak C}\mathcal H$).
Since $M_{\mathfrak C}\mathcal H=C_{w_{J(\mu)}}\mathcal H$ and $S^\lambda\mathcal H=x_\lambda T_{w_E}C_{w_{J(\mu)}}\mathcal H$
we can see that $\theta$ is a surjective $\mathcal H$-module homomorphism.
Invoking Result~\ref{Result:JPAATh3.5} we get that $\hat M_{\mathfrak C}\subseteq\ker\theta$.
Since $\hat M_{\mathfrak C}\mathcal H=\langle C_yT_v:\ y<_R'\mathfrak C,\ v\in\mathfrak X' \rangle_A$, we also get that $\hat M_{\mathfrak C} \mathcal H\subseteq\ker\theta$.
It follows that $S^\lambda \mathcal H$ is spanned over $A$ by the set $\mathcal B$.
To show that $\mathfrak B$ is $A$-linearly independent, we temporarily extend scalars from $A$ to $F$ and consider the (surjective) $\mathcal{H}_F$-module homomorphism $\theta_F:C_{w_{J(\mu)}}\mathcal{H}_F\to x_{\lambda}T_{w_E}C_{w_{J(\mu)}}\mathcal{H}_F\,(=S_F^\lambda\mathcal H_F=S^\lambda \mathcal H_F)$ given by premultiplication by $x_\lambda T_{w_E}$.
The above discussion ensures that $S_F^\lambda\mathcal H_F$ is spanned, over $F$, by the set $\mathcal B$.
Now $|\mathfrak X'|=n+1$, so $|\mathcal B|=(n+1)|\mathfrak C|$.
But $(n+1)|\mathfrak C|=\dim_F(S_F^\lambda\mathcal H_F)$ since $|\mathfrak C|=\dim_FS_F^\lambda=\dim_FS_F^\lambda\mathcal H_F'$ (see~\cite[Theorem~4.6]{MPa15}) and $S_F^\lambda\mathcal H_F\cong S_F^\lambda\uparrow\mathcal H_F$.
It follows that the set $\mathcal B$ is $F$-linearly independent and hence $A$-linearly independent.

We conclude that $S^\lambda\mathcal H$ is a free $A$-module having the set $\mathcal B$ as an $A$-basis.

It is also immediate from the above that $\overline{L}_p=S^\lambda\mathcal H$ and that, for $1\le i\le p$, the set $\mathcal B_i=\{x_{\lambda}T_{w_E}C_w: w\in\cup_{1\le j\le i} \mathfrak{C_j}\}$ is an $A$-basis for $\overline{L}_i$.
At this point it would be useful to recall (see the proof of Proposition~\ref{prop:4.1c}) that $i<j$ whenever $\mathfrak{C_i}<_R\mathfrak{C_j}$ for $1\le i,j\le p$.
Now fix $i$ with $1\le i\le p$.
Also let $h\in\mathcal H$ and let $w\in\cup_{1\le j\le i} \mathfrak{C_j}$.
Note that $C_w\in M_{\mathfrak{C}}\mathcal H\,(=\langle C_{yv}:\ y\LEQ_R'\mathfrak{C}, v\in\mathfrak X' \rangle_A)$.
Since $M_{\mathfrak{C}}\mathcal H$ is an $\mathcal H$-module, $C_wh\in M_{\mathfrak{C}}\mathcal H$.
It follows that $C_wh$ can be expressed as a sum $\Sigma_1+\Sigma_2+\Sigma_3$ where $\Sigma_1$ is an $A$-linear combination of terms $C_z$ with $z\in\mathfrak C_i$, $\Sigma_2$ is an $A$-linear combination of terms $C_z$ with $z\in\cup_{1\le j\le i-1}\mathfrak C_j$ (in particular, $\Sigma_2$ is an empty sum if $i=1$) and $\Sigma_3\in\hat M_{\mathfrak{C}}\mathcal H$.
As we have already observed $\hat M_{\mathfrak{C}}\mathcal H\subseteq\ker\theta$ so $x_{\lambda}T_{w_E}\Sigma_3=0$.
It follows that $\overline{L}_i$ is indeed an $\mathcal H$-submodule of $S^\lambda\mathcal H$ and clearly $\overline{L}_{i-1}\subseteq\overline{L}_i$.
Moreover, $\overline{L}_i/\overline{L}_{i-1}\cong S_{\mathfrak{C}_i}$ as $\mathcal H$-modules since the natural map $\sum_{u\in\mathfrak{C}_i}\xi_uC_u+\hat M_{\mathfrak{C}_i}\mapsto x_{\lambda}T_{w_E}(\sum_{u\in\mathfrak{C}_i}\xi_uC_u)+\overline{L}_{i-1}$ from $S_{\mathfrak{C}_i}$ to $\overline{L}_i/\overline{L}_{i-1}$ is clearly an $\mathcal H$-module isomorphism (recall that $\hat M_{\mathfrak{C}_i}\subseteq L_{i-1}$ with $L_{i-1}$ as defined in the proof of Proposition~\ref{prop:4.1c}).
Finally, our assumption that $\shape\mathfrak{C}_i=\nu^{(i)}$ ensures, in view of Remark~\ref{RemarkSpecht}, that $\overline{L}_i/\overline{L}_{i-1}\, (\cong S_{\mathfrak{C}_i})\cong S^{\nu^{(i)}}$.
By construction, $\shape\mathfrak{C}_i\lhd\shape\mathfrak{C}_j$ if $i<j$.
\end{proof}

\begin{remark}
(i)
Under the hypothesis of Theorem~\ref{thm:SpFilForS2H} and keeping the notation we have introduced therein, we can say a bit more about the surjective $\mathcal H$-module homomorphism from $M_{\mathfrak{C}}\mathcal H$ to $S^\lambda\mathcal H$ which maps $m\in M_{\mathfrak{C}}\mathcal H$ to $x_{\lambda}T_{w_E}m$.
We have already shown that $\hat M_{\mathfrak{C}}\mathcal H\subseteq\ker\theta$.
In fact, $\hat M_{\mathfrak{C}}\mathcal H=\ker\theta$.
To show that the reverse inclusion $\ker\theta\subseteq\hat M_{\mathfrak{C}}\mathcal H$ also holds, we argue as in the proof of~\cite[Theorem~3.5]{MPa05}.
So let $r\in\ker\theta$.
Then $r=m+\sum_{y\in\mathfrak{C},v\in\mathfrak X'}\alpha_{y,v}C_{yv}$ for some $m\in\hat M_{\mathfrak{C}}\mathcal H$ and $\alpha_{y,v}\in A$.
Hence, $0=r\theta=m\theta+\sum_{y\in\mathfrak{C},v\in\mathfrak X'}\alpha_{y,v}(C_{yv}\theta)$.
Since $m\in\ker\theta$ we get $\sum_{y\in\mathfrak{C},v\in\mathfrak X'}\alpha_{y,v}(C_{yv}\theta)=0$.
But we have already shown in the proof of Theorem~\ref{thm:SpFilForS2H} that the set $\{C_{yv}\theta:\ y\in\mathfrak{C},v\in\mathfrak X'\}\, (=\{C_w\theta:$ $w\in\cup_{1\le j\le p}\mathfrak{C}_j\})$ is $A$-linearly independent.
Hence $\alpha_{y,v}=0$ for all $y\in\mathfrak{C}$ and $v\in\mathfrak X'$.
This ensures that $r=m\in\hat M_{\mathfrak{C}}\mathcal H$ as required.
We conclude that the map $\theta$ induces a natural $\mathcal H$-isomorphism from the induced cell module $S_{\mathfrak{C}}\uparrow\mathcal H$ which is isomorphic to $M_{\mathfrak{C}}\mathcal H/\hat M_{\mathfrak{C}}\mathcal H$, to the induced Specht module $S^\lambda\uparrow\mathcal H$, which is isomorphic to $S^\lambda\mathcal H$.

(ii)
For a fixed composition $\lambda$ of $n$, the construction in the theorem above gives for each composition $\mu$ of $n$ with $\lambda''=\mu'$ a different $A$-basis for $S^\lambda\mathcal H$ thus giving rise to a different Specht filtration for this module.
\end{remark}

For the next result we keep the notation in Result~\ref{prop:2.2c} and Proposition~\ref{prop:4.3c}.

\begin{theorem}\label{thm:SpFilForSlambdaRestToH}
Let $\lambda,\mu\vDash n+1$ with $\lambda''=\mu'$, let $\mathfrak C$ be the right cell of $W$ containing $w_{J(\mu)}$ and let $D$ be the underlying diagram of the recording tableau of the elements of $\mathfrak C$ (so $D$ is a $\lambda''$ diagram).
Also let $E$ be the unique diagram in $\mathcal D^{(\lambda,\mu)}$.
Write ${\rm ic}(D)=\{k_1,\ldots,k_p\}$, where $k_i<k_j$ if $i<j$.
Also write $\mathfrak C_i$ and $d_i$ in place of $\mathfrak C_{k_i}$ and $d_{k_i}$ respectively (so $\mathfrak C=\cup_{1\le i\le p}d_i \mathfrak C_i$).
Finally, for $1\le i\le p$, let $\overline{L}_i=\langle x_{\lambda}T_{w_E}C_w:\ w\in\cup_{1\le j\le i} d_j\mathfrak C_j\rangle_A$.
Then, $\{0\}=\overline L_0\subseteq\overline L_1\subseteq\ldots\subseteq\overline L_p=S^\lambda\downarrow\mathcal H'$ is a Specht filtration (of $\mathcal H'$-submodules of $S^\lambda$) for the $\mathcal H$-Specht module $S^\lambda$ restricted to $\mathcal H'$ with $\overline L_i/\overline L_{i-1}\cong S^{\nu^{(i)}}$, where $\nu^{(i)}=\shape\mathfrak C_i$ for $1\le i\le p$.
(In particular, $\nu^{(i)}\lhd\nu^{(j)}$ if $i<j$.)
\end{theorem}

\begin{proof}
Assume the hypothesis.
For $1\le i\le p$ set $\mathcal B_i=\{x_{\lambda}T_{w_E}C_w: w\in\cup_{1\le j\le i} d_j\mathfrak{C_j}\}$.
From Result~\ref{Result:JPAATh3.5}, $\mathcal B_p$ is an $A$-basis for $S^\lambda$.
It follows that $\mathcal B_i$ is an $A$-basis for $\overline L_i$ for $1\le i\le p$.
Clearly $S^\lambda=\overline L_p$ as sets.
Arguing as in the proof of Proposition~\ref{prop:4.3c}, see in particular equation~\eqref{eq:(6)} and the discussion following it, and invoking Result~\ref{Result:JPAATh3.5} again, we can deduce that, for $1\le i\le p$, $\overline L_i$ is an $\mathcal H'$-submodule of $S^\lambda$ with $\overline L_i/\overline L_{i-1}\cong S_{\mathfrak C_i}\cong S^{\nu^{(i)}}$ (recall $\nu^{(i)}=\shape\mathfrak C_i$).
Also observe that $\overline L_p$, as an $\mathcal H'$-module is in fact $S^\lambda\downarrow\mathcal H'$.
The relation in the final sentence of the statement of the theorem is immediate from Result~\ref{prop:2.2c}.
\end{proof}

\section{A link with sequences and pairs of partitions}\label{sec:LinkSequenPairsOfPart}

In this section we provide a link between a certain $C$-basis for the induced Specht module considered earlier on in this paper with the notion of pairs of partitions introduced by James in~\cite{James1977} and which he also describes in~\cite{Jam78}.
We refer the reader to~\cite{Jam78} for the basic definitions and background concerning sequences, semistandard tableaux and pairs of partitions which will be needed in this section.
Note that in~\cite{Jam78} (see also~\cite{DJa86} for the Hecke algebra case), the notion of pair of partitions is used in order to construct Specht filtrations for certain modules, denoted by $S^{\mu^{\sharp},\mu}$ in these papers, which are generalizations of induced Specht modules associated to partitions.


We begin by describing another useful way of comparing compositions of not necessarily the same integer $m$ (compare with definition of pairs of partitions~\cite[page~54]{Jam78}):
Let $\lambda=(\lambda_i)\vDash m_1$ and $\nu=(\nu_i)\vDash m_2$.
We write $\lambda\preceq\nu$ if $\lambda_i\le\nu_i$ for all $i$ (so in such a case we have $m_1\le m_2$).

Throughout this section we fix $\mu=(\mu_1,\ldots,\mu_r)\vDash m$.

A (finite) sequence is said to have type $\mu$ if, for each $i$, with $1\le i\le r$, symbol $i$ occurs $\mu_i$ times in the sequence (see~\cite[page~54]{Jam78}).
We denote by $\mathfrak R(\mu)$ the set of sequences of type $\mu$.
Also let $L(\mu)=\{w\in S_m:\ w\LEQ_Lw_{J(\mu)}\}$.
Recall that $L(\mu)=\mathfrak X_{J(\mu)}^{-1}w_{J(\mu)}$ (this follows, for example, from~\cite[Proposition~2.4]{KLu79} in view of~\cite[Proposition~1.5.1(c), Lemma~1.5.2 and Proposition~2.1.1]{GPf00}).
Finally, for $1\le i\le r$, we denote by $d^\mu(i)$ the decreasing sequence of the $\mu_i$ successive numbers starting with $\mu_1+\ldots+\mu_{i-1}+\mu_i$ and ending with $\mu_1+\ldots+\mu_{i-1}+1$ (in particular $d^\mu(1)$ is the decreasing sequence $\mu_1,\ldots,1$).
Note that the row-form of $w_{J(\mu)}$ can be obtained by placing $d^\mu(1)$, $d^\mu(2)$, \ldots, $d^\mu(r)$ one next to the other in that order.

We now define a map $\mathfrak R(\mu)\to L(\mu)$ by $\mathfrak t\mapsto w(\mathfrak t)$ ($\mathfrak t\in\mathfrak R(\mu)$), where $w(\mathfrak t)$ is the element of $S_m$ obtained from $\mathfrak t$ by replacing, for $i=1,\ldots,r$\,, the $i$'s in $\mathfrak t$, going through the sequence from left to right, by the members of $d^\mu(i)$, in the order they appear in this decreasing subsequence.

Clearly for $\mathfrak t_1,\mathfrak t_2\in\mathfrak R(\mu)$ with $\mathfrak t_1\ne\mathfrak t_2$, we have $w(\mathfrak t_1)\ne w(\mathfrak t_2)$.
Also observe that each one of the $d^\mu(i)$ occurs as a decreasing subsequence in the row-form of $w(\mathfrak t)$ whenever $\mathfrak t\in\mathfrak R(\mu)$.
Now $L(\mu)\,(=\mathfrak X_{J(\mu)}^{-1}w_{J(\mu)})$ consists of precisely those elements of $S_m$ which have each one of the $d^\mu(i)$, for $1\le i\le r$, appearing as a decreasing subsequence in their row-form.
This establishes that the above map from $\mathfrak R(\mu)$ to $L(\mu)$ is well-defined and bijective.

As in~\cite[\S13]{Jam78}, for the discussion that follows it will be convenient to introduce tableaux $T$ having repeated entries.
Such tableaux will be denoted by capital letters.
A tableau $T$ is said to have type $\mu\,(=(\mu_1,\ldots,\mu_r)$), if for every $i$ the number $i$ occurs $\mu_i$ times in $T$.

In analogy with the definition of a semistandard tableau in~\cite[page~45]{Jam78} we say that a tableau $T$ is \emph{$c$-semistandard} if the numbers are non-decreasing down the columns of $T$ and strictly increasing along the rows of $T$.
For $\lambda\vdash m$ we denote by $\mathcal T_c(\lambda,\mu)$ the set of $c$-semistandard $\lambda$-tableaux of type $\mu$.
We also let $\mathcal T_c(\mu)=\cup_{\lambda\vdash m}\mathcal T_c(\lambda,\mu)$.

We can make the following observation.
Suppose $T\in\mathcal T_c(\mu)$ (with $\mu=(\mu_1,\ldots,\mu_r)\vDash m$).
Then symbol $i$ cannot appear in any column of $T$ which is to the right of the $i$-th column of $T$.
In particular, $T$ has at most $r$ columns.

\begin{definition}\label{def:WordOfTPT}
Let $T\in\mathcal T_c(\mu)$.

(i) We define the \emph{word} of $T$ to be the sequence $\mathfrak t_T\in\mathfrak R(\mu)$ obtained from $T$ going through the entries of $T$ as follows:
Starting from the first column, write down the entries from bottom to top, then list the entries from bottom to top in the second column, working across to the last column (compare~\cite[page~17]{Ful97}).
It is easy to see that a tableau $T\in\mathcal T_c(\mu)$ can easily be recovered from its word $\mathfrak t_T$.

(ii) We define $P_T$ to be the insertion tableau for the element $w(\mathfrak t_T)\,(\in L(\mu))$ in the Robinson--Schensted process, that is, $P_T=P(w(\mathfrak t_T))$.
\end{definition}

Given $T\in\mathcal T_c(\mu)$, it follows easily from the fact that $T$ is $c$-semistandard that  $P_T$ can be obtained from $T$ by replacing each entry of $T$ with the same number with which we need to replace the corresponding entry in $\mathfrak t_T$ in order to obtain $w(\mathfrak t_T)$.
We can say a bit more:
Again from the fact that $T$ is $c$-semistandard we know that symbol $i$, for $1\le i\le r$, appears at most once in a row of $T$.
It follows that $P_T$ is in fact obtained from $T$ by replacing symbol $i$, as we go from the bottom to the top row of $T$, by the members of $d^\mu(i)$, in the order they appear in $d^\mu(i)$.

Recall that $w(\mathfrak t_T)\in L(\mu)$ whenever $T\in\mathcal T_c(\mu)$.
Suppose now that $T_1,T_2\in\mathcal T_c(\mu)$ with $T_1\ne T_2$.
Then clearly $P_{T_1}\ne P_{T_2}$.
But $P_{T_1}=P(w(\mathfrak t_{T_1}))$ and $P_{T_2}=P(w(\mathfrak t_{T_2}))$.
It follows that $w(\mathfrak t_{T_1})$ and $w(\mathfrak t_{T_2})$ belong to different left cells inside the union of left cells $L(\mu)$.
In other words, the correspondence $T\mapsto P_T$ induces an injection from $\mathcal T_c(\mu)$ to the set of left cells contained in the union of left cells $L(\mu)$.
The following result, which is an immediate consequence of Young's rule, will ensure that the correspondence $T\mapsto P_T$ in fact induces a bijection between the above sets.

\begin{proposition}\label{prop:4.2}
Let $\mu\vDash m$ and let $\lambda\vdash m$.
Then, the number of left cells $\mathfrak C$ of $W=S_m$ satisfying $\mathfrak C\LEQ_Lw_{J(\mu)}$ and $\shape\mathfrak C=\lambda$ equals the number of $c$-semistandard $\lambda$-tableaux of type $\mu$.
\end{proposition}

\begin{proof}
We assume the hypothesis.
As in~\cite{MPa05}, for $x\in W=S_m$, we denote by $S_{F,x}^{\bullet}$ the ${\mathcal H}$-module with $F$-basis $\{C_z'+(\hat M_xj_{\mathcal H})_F:\ z\sim_Rx\}$.
(Note that $(\hat M_xj_{\mathcal H})_F=\hat M_{F,x}j_{{\mathcal H}_F}$ since $j_{\mathcal H}$ can be extended uniquely to an automorphism $j_{{\mathcal H}_F}$ of $\mathcal H_F$.)
We aim to show that the Specht module $S_F^{\lambda'}$ is $\mathcal H_F$-isomorphic to $S_{F,w_{J(\lambda')}}^{\bullet}$.
For this, we define $w=w_0w_{J(\lambda)}w_{\lambda}$ as in~\cite[Lemma~3.3]{MPa05}
(by $w_0$ we denote the longest element of $W=S_m$ and by $w_\lambda$ the element $w_D$ where $D$ is the unique diagram in $\mathcal D^{(\lambda,\lambda')}$).
In view of Result~\ref{res:5c}, we get from~\cite[Lemma~3.3(v)]{MPa05} that $\shape w=\shape w_{J(\lambda')}$.
From~\cite[\S5]{KLu79} (see also~\cite[Corollary~5.8]{Gec05}) we get that $S_{F,w}^{\bullet}\cong S_{F,w_{J(\lambda')}}^{\bullet}$ as $\mathcal H_F$-modules.
From~\cite[Lemma~3.4]{MPa05} and its proof we also know that $S_{F,w}^{\bullet}\cong S_{F,w_{J(\lambda)}}$ and that $S_{F,w_{J(\lambda)}}\cong S_F^{\lambda'}$ as $\mathcal H_F$-modules.
We conclude that $S_{F,w_{J(\lambda')}}^{\bullet}\cong S_{F,w}^{\bullet}\cong S_{F,w_{J(\lambda)}}\cong S_F^{\lambda'}$ as $\mathcal H_F$-modules.

For the remainder of the proof we work under the specialization $q^{1/2}\mapsto1$ in order to apply Young's rule (see, for example,~\cite[page~51]{Jam78}).
Let $l$ be the number of $c$-semistandard $\lambda$-tableaux of type $\mu$.
Then clearly $l$ is the number of semistandard $\lambda'$-tableaux of type $\mu$.
From Young's rule we get that $l$ is the number of times $S_F^{\lambda'}\,(\cong S_{F,w_{J(\lambda')}}^{\bullet})$ occurs as a composition factor of $(M_{w_{J(\mu)}}j_{\mathcal H})_F$.

We conclude, invoking again the result in~\cite{KLu79} (or~\cite{Gec05}) used earlier in the proof, that $l$ is the number of right cells $\mathfrak C$ with $\shape\mathfrak C=\lambda$ occurring in $\{w\in W:\ w\LEQ_Rw_{J(\mu)}\}$ which is the number of left cells $\tilde{\mathfrak C}$ with $\shape\tilde{\mathfrak C}=\lambda$ occurring in $L(\mu)$.
\end{proof}

\begin{corollary}\label{cor:4.3}
Let $w\in S_m$.
Then $w\LEQ_L w_{J(\mu)}$ if, and only if, $P(w)=P_T$ for some $T\in\mathcal T_c(\mu)$.
\end{corollary}

At this point we need to recall some definitions in~\cite{Jam78}.

\begin{definition}
(i) \cite[Definition 15.2]{Jam78}
Given $\mathfrak t\in\mathfrak R(\mu)$, the quality of each term (either good or bad) of $\mathfrak t$ is determined as follows.\\
(a) All the 1's are good.\\
(b) An $(i+1)$ is good if, and only if, the number of previous good $i$'s is strictly greater than the number of previous good $(i+1)$'s.

(ii) \cite[Definitions 15.5 and 15.6]{Jam78}
Let $\lambda=(\lambda_1,\lambda_2,\ldots)$ be a sequence of non-negative integers such that for all $i$, $\lambda_{i+1}\le \lambda_i\le \mu_i$.
Then $(\lambda,\mu)$ is called a \emph{pair of partitions} for $m$.
In such a case we denote by $\mathfrak R(\lambda,\mu)$ the subset of $\mathfrak R(\mu)$ consisting of those sequences of type $\mu$ in which for every $i$, the number of good $i$'s is at least $\lambda_i$.

(iii) We define the \emph{sharp partition}, $\mu^{\sharp}(\mathfrak t)$, of $\mathfrak t\in\mathfrak R(\mu)$ by $\mu^{\sharp}(\mathfrak t)=(\mu^{\sharp}_i)$, where, for $i=1,2,\ldots$, the part $\mu^{\sharp}_i$ equals the number of good $i$'s in $\mathfrak t$.
(Then $\mu^{\sharp}(\mathfrak t)$ is a partition of $k$ for some $k\le m$.
Clearly $\mu^{\sharp}(\mathfrak t)\preceq\mu$, alternatively we can say that $(\mu^{\sharp},\mu)$ is a pair of partitions for~$m$.)
\end{definition}

\begin{remark}\label{rem:4.5}
Suppose $(\lambda,\mu)$ is a pair of partitions for $m$ (recall that $\mu=(\mu_1,\ldots,\mu_r)\vDash m$), with $\lambda=(\lambda_i)$.
Suppose further that $\mathfrak t\in\mathfrak R(\mu)$.
Then $\mu^{\sharp}(\mathfrak t)=\lambda$ if, and only if, for each $i$ there exist precisely $\lambda_i$ members of decreasing sequence $d^\mu(i)$ in the $i$-th column of $P(w(\mathfrak t))$.
Suppose now that $\mu^{\sharp}(\mathfrak t)=\lambda$.
We can then observe that the $\lambda_i$ members of $d^\mu(i)$ which appear in the $i$-th column of $P(w(\mathfrak t))$ must be contained in the top $\lambda_i$ rows of $P(w(\mathfrak t))$ and are necessarily the smallest $\lambda_i$ members of $d^\mu(i)$.
(Recall that $d^\mu(i)$ has precisely $\mu_i$ members.)
\end{remark}

The remaining results of this section give a connection between the notion of pair of partitions and Kazhdan--Lusztig cells.

\begin{proposition}\label{prop:LmulambdaLlambdamu}
Suppose that $(\lambda,\mu)$ is a pair of partitions for $m$.
Define the subsets $L(\mu;\lambda)$ and $L(\lambda,\mu)$ of $L(\mu)$ by
$L(\mu;\lambda)=\{w\in W:$ $w=w(\mathfrak t)$ for some $\mathfrak t\in\mathfrak R(\mu)$ with $\mu^{\sharp}(\mathfrak t)=\lambda\}$ and
$L(\lambda,\mu)=\{w\in W:$ $w=w(\mathfrak t)$ for some $\mathfrak t\in\mathfrak R(\lambda,\mu)\}$.
Then each one of the sets $L(\mu;\lambda)$ and $L(\lambda,\mu)$ is a union of left cells of $W$.
\end{proposition}

\begin{proof}
Assume the hypothesis and let $\lambda=(\lambda_i)$.
It follows from Corollary~\ref{cor:4.3} and Remark~\ref{rem:4.5} that $L(\mu;\lambda)=\{w\in W:$ $P(w)=P_T$ for some $T\in\mathcal T_c(\mu)$ such that $\lambda_i$ is the number of times $i$ appears in the $i$-th column of $T\}=\{w\in W:$ $P(w)=P_T$ for some $T\in\mathcal T_c(\mu)$ such that $\lambda_i$ is the number of members of $d^\mu(i)$ in the $i$-th column of $P(w)\}$.

Moreover, $L(\lambda,\mu)=\{w\in W:$ $w=w(\mathfrak t)$ for some $\mathfrak t\in\mathfrak R(\mu)$ such that $\lambda\preceq\mu^{\sharp}(\mathfrak{t})\preceq\mu\}=\cup_{\lambda\preceq\nu\preceq\mu}L(\mu;\nu)$.

The required result follows easily from the above observations in view of Result~\ref{res:5c}.
\end{proof}

\begin{remark}\label{rem:4.7}
(i) The proof of Proposition~\ref{prop:LmulambdaLlambdamu}, compare also with Remark~\ref{rem:4.5}, gives some additional information to that actually given in the statement of this proposition, as it provides a precise description of all the left cells (via the Robinson-Schensted insertion tableau of their elements) occurring in $L(\mu;\lambda)$ and $L(\lambda,\mu)$.

(ii) Keeping the notation introduced in Proposition~\ref{prop:LmulambdaLlambdamu} we can also observe that
$L(\lambda,\mu)=\{w\in W:$ $P(w)=P_T$ for some $T\in\mathcal T_c(\mu)$ such that the number of $i$'s in the $i$-th column of $T$ is at least $\lambda_i\}=\{w\in W:$ $P(w)=P_T$ for some $T\in\mathcal T_c(\mu)$ such that the number of members of $d^\mu(i)$ in the $i$-th column of $P(w)$ is at least $\lambda_i\}$.
\end{remark}

\begin{theorem}\label{prop:(1)}
Let  $\mathfrak t\in\mathfrak R(\mu)$ and let $e\in\mathfrak X_{J(\mu)}$ be defined by $e^{-1}w_{J(\mu)}=w(\mathfrak t)\,(\in L(\mu))$.
Suppose further that $d$ is a prefix of $e$ and let  $\mathfrak t'\in\mathfrak R(\mu)$ be defined by $d^{-1}w_{J(\mu)}=w(\mathfrak t')$.
Then $\mu^{\sharp}(\mathfrak t)\preceq\mu^{\sharp}(\mathfrak t')\,(\preceq\mu)$.
\end{theorem}

\begin{proof}
Assume the hypothesis.
It suffices to consider the case $e=ds$ with $l(e)=l(d)+1$ (where $s$ is a Coxeter generator, say $s=(k\ k+1)$).
It follows that $w(\mathfrak t)=e^{-1}w_{J(\mu)}=sd^{-1}w_{J(\mu)}=sw(\mathfrak t')$.
Hence the row-form of $w(\mathfrak t)$ is obtained from the row-form of $w(\mathfrak t')$ by interchanging the entries in the $k$-th and $(k+1)$-th positions of $w(\mathfrak t')$.
Suppose $i$ and $j$ are respectively the entries at the $k$-th and $(k+1)$-th positions of $w(\mathfrak t')$.
The fact that $e=ds$ with $l(e)=l(d)+1$ with $e,d\in\mathfrak X_{J(\mu)}$ ensures that $l(w(\mathfrak t))=l(w(\mathfrak t'))+1$ and hence that $i<j$.
Moreover, since $e,d\in\mathfrak X_{J(\mu)}$, there exist $p,q\in\{1,\ldots,r\}$ with $p<q$ such that $i$ occurs in sequence $d^\mu(p)$ and $j$ occurs in sequence $d^\mu(q)$.
Clearly $\mu^{\sharp}(\mathfrak t)=\mu^{\sharp}(\mathfrak t')$ if $q>p+1$.
So for the rest of the proof we assume that $q=p+1$.
We first consider the case where the $p$ in the $k$-th position of $\mathfrak t'$ is bad.
Interchanging the entries in the $k$-th and $(k+1)$-th positions of $\mathfrak t'$ then has no effect on the quality of these entries showing that, in this case, $\mu^{\sharp}(\mathfrak t)=\mu^{\sharp}(\mathfrak t')$.
Next we consider the case where the $p$ in the $k$-th position of $\mathfrak t'$ is good.
If the $q$ in the $(k+1)$-th position of $\mathfrak t'$ is bad it is again clear that $\mu^{\sharp}(\mathfrak t)=\mu^{\sharp}(\mathfrak t')$.
Thus, in order to complete the proof, it is enough to consider the case where the $p$ in the $k$-th position and the $q\,(=p+1)$ in the $(k+1)$-th position of $\mathfrak t'$ are both good.

Let $l_1$ and $l_2$ be respectively the number of good $p$'s and the number of good $q$'s occurring in the first $k+1$ entries of $\mathfrak t'$.
Then, for the case we are considering, we have $l_1\ge l_2$.
If $l_1>l_2$, it is clear that interchanging the entries in the $k$-th and $(k+1)$-th positions of $\mathfrak t'$ does not affect the quality of any of the terms of $\mathfrak t'$, so again we have $\mu^{\sharp}(\mathfrak t)=\mu^{\sharp}(\mathfrak t')$ .
We can thus assume that $l_1=l_2$.
It follows from this last assumption that by interchanging the entries in the $k$-th and $(k+1)$-th positions of $\mathfrak t'$, the `good' $q$ which was in the $(k+1)$-th position now becomes a `bad' $q$ in the $k$-th position
(the quality of the $p$ which moves from position $k$ to position k+1 is not affected).
Now let $k'\ge k+1$.
Then the above discussion ensures that the number of good $q$'s in the first $k'$ positions of $\mathfrak t$ is less than or equal to the number of good $q$'s in the first $k'$ positions of $\mathfrak t'$.
Similarly, if we consider successively the way in which the quality of the entries $q+1$, $q+2$, \ldots\ in $\mathfrak t'$ can be affected by interchanging the entries in the $k$-th and $(k+1)$-th positions of $\mathfrak t'$, we can see that this last statement is in fact true with $q_1$ in the place of $q$ for any $q_1\ge q$.
We conclude that $\mu^{\sharp}(\mathfrak t)\preceq\mu^{\sharp}(\mathfrak t')$ in this final case also.
\end{proof}

\begin{corollary}\label{cor:4.9}
Suppose $(\lambda,\mu)$ is a pair of partitions for $m$ and let  $e\in\mathfrak X_{J(\mu)}$.
If $e^{-1}w_{J(\mu)}\in L(\lambda,\mu)$, then $d^{-1}w_{J(\mu)}\in L(\lambda,\mu)$ for any prefix $d$ of $e$.
\end{corollary}

Next we list some consequences of Proposition~\ref{prop:LmulambdaLlambdamu}  for some special cases of~$\mu$.

\begin{proposition}\label{prop:(2)}
If $\mu$ is a partition of $m$, then $L(\mu,\mu)=\{w\in S_m:\ w\sim_Lw_{J(\mu)}\}$.
\end{proposition}

\begin{proof}
Immediate from Proposition~\ref{prop:LmulambdaLlambdamu} (see also Remark~\ref{rem:4.7}) in view of Result~\ref{res:5c}.
\end{proof}

\begin{proposition}\label{prop:(3)}
Suppose $m>1$, $r>1$, $\mu_r=1$ and that $\mu\,(=(\mu_1,\ldots,\mu_{r-1},1))$ is a partition of $m$.
Let $\lambda$ be the partition $(\mu_1,\ldots,\mu_{r-1})$ of $m-1$, $\mathfrak C$ the left cell containing $w_{J(\lambda)}$ and $\mathfrak X^*$ the set of distinguished left coset representatives of $S_{m-1}$ in $S_m$.
Then $L(\lambda,\mu)=\mathfrak X^*\mathfrak C$.
\end{proposition}

\begin{proof}
Immediate from Proposition~\ref{prop:LmulambdaLlambdamu} (see also Remark~\ref{rem:4.7}) in view of Result~\ref{prop:2.1c}.
\end{proof}

\begin{remark}
(i) The last two results give a connection between sequences and pairs of partitions and certain $C$-bases for the Specht module and the induced Specht module we have described earlier on (see, in particular, Result~\ref{Result:JPAATh3.5} and the proof of Theorem~\ref{thm:SpFilForS2H}).

(ii) It follows from Theorem~\ref{prop:(1)} that the set $T=\{e\in\mathfrak X_{J(\mu)}: e^{-1}w_{J(\mu)}\in L(\lambda,\mu)\}$ satisfies the Schreier property (that is, it contains all prefixes of all its elements).
In both of the special  cases considered in Propositions~\ref{prop:(2)} and~\ref{prop:(3)} we also have, in view of the results in Section~\ref{Section:SpechtFiltrations}, that the set $T_1=L(\mu)\setminus\{e^{-1}w_{J(\mu)}:\ e\in T\}$ has the property that $x\in T_1$ whenever $y\in T_1$ and $x\LEQ_Ly$.
It would be interesting to investigate whether this last property  still holds in the more general case of $\lambda$ and $\mu$ considered in Theorem~\ref{prop:(1)}.
\end{remark}

\subsection*{Acknowledgment.} I would like to thank Thomas McDonough for very useful discussions and comments.

\end{document}